\documentclass[10pt, reqno]{amsart}
\usepackage{amssymb,amsmath,amsthm,tikz}
\usetikzlibrary{tikzmark,calc}

\usepackage{graphicx}
\usepackage{xcolor}

\usepackage[normalem]{ulem}

\usepackage{hyperref}

\theoremstyle{theorem}
\newtheorem{theorem}{Theorem}[section]

\newtheorem{corollary}[theorem]{Corollary}
\newtheorem{lemma}[theorem]{Lemma}

\newtheorem{remark}[theorem]{Remark}

\theoremstyle{definition}

\theoremstyle{example}

\newcommand{\boundellipse}[3]
{(#1) ellipse (#2 and #3)
}

\allowdisplaybreaks

\title[]{A Family of Congruences Modulo 7 for Partitions with Monochromatic Even Parts and Multi--Colored Odd Parts}

\author{Michael D. Hirschhorn}
\address{School of Mathematics, UNSW, Sydney, 2052, Australia}
\email{m.hirschhorn@unsw.edu.au}

\author{James A. Sellers}
\address{Department of Mathematics and Statistics, 
University of Minnesota Duluth, Duluth, MN 55812, USA}
\email{jsellers@d.umn.edu}



\date{}

\begin{document}


\maketitle

\begin{abstract}
In recent work, Amdeberhan and Merca considered the integer partition function $a(n)$ which counts the number of integer partitions of weight $n$ wherein even parts come in only one color (i.e., they are monochromatic), while the odd parts may appear in one of three colors.  One of the results that they proved was that, for all $n\geq 0$, $a(7n+2) \equiv 0 \pmod{7}$.  In this work, we generalize this function $a(n)$ by naturally placing it within an infinite family of related partition functions.  Using elementary generating function manipulations and classical $q$--series identities, we then prove infinitely many congruences modulo 7 which are satisfied by members of this family of functions.  
\end{abstract}

\section{Introduction and background}
\label{sec:intro}
A partition of a positive integer $n$ is a finite sequence of positive integers $\lambda = (\lambda_1, \ldots, \lambda_j)$ with $\lambda_1 + \dots + \lambda_j = n$.
The $\lambda_i$, called the parts of $\lambda$, are ordered so that 
\begin{equation}
\label{ordinary_ordering}
\lambda_1 \ge \cdots \ge \lambda_j.
\end{equation}  We denote the number of partitions of $n$ by $p(n)$; for example, the partitions of $n=4$ are 
$$(4), \ \ (3,1), \ \ (2,2), \ \ (2,1,1), \ \ (1,1,1,1),$$
which means $p(4) = 5$.  More information about integer partitions can be found in \cite{AE}.  

Recently, Amdeberhan and Merca \cite{AM} considered the function $a(n)$ which is defined by 
\begin{equation}
\label{genfn1}
\sum_{n=0}^\infty a(n)q^n = \frac{(-q;q)_\infty^2}{(q;q)_\infty}
\end{equation}
where $(A;q)_\infty = (1-A)(1-Aq)(1-Aq^2)(1-Aq^3)\dots $ is the usual $q$--Pochhammer symbol.  It is straightforward to see that \eqref{genfn1} can be rewritten as 
\begin{equation}
\label{genfn2}
\sum_{n=0}^\infty a(n)q^n =  \frac{(-q;q)_\infty^2}{(q;q)_\infty}\cdot \frac{(q;q)_\infty^2}{(q;q)_\infty^2} = \frac{(q^2;q^2)_\infty^2}{(q;q)_\infty^3} = \frac{f_2^2}{f_1^3}
\end{equation}
where we have used the shorthand notation $f_k:=(q^k;q^k)_\infty$.  It is also clear that 
$$
\frac{(q^2;q^2)_\infty^2}{(q;q)_\infty^3} = \frac{1}{(q^2;q^2)_\infty (q;q^2)_\infty^3};
$$ 
from this generating function identity we see that $a(n)$ counts the number of integer partitions of weight $n$ wherein even parts come in only one color (i.e., they are monochromatic), while the odd parts may appear in one of three colors \cite[\href{https://oeis.org/A298311}{A298311}]{OEIS}.  (This interpretation of $a(n)$, along with a few others, appears in the paper of Amdeberhan and Merca \cite{AM}.)  Thus, for example, a part of the form $5_3$ in such a partition means the part has weight 5 and is endowed with the color 3.  Within this context, the ordering of the parts \eqref{ordinary_ordering} mentioned above is no longer sufficient (as it does not take into account the possible colors).  Thus, to construct these ``colored'' partitions, we define the following ordering on colored parts:
$ \alpha_{c_i} > \beta_{c_j}$  exactly when $\alpha > \beta$ or when $ \alpha=\beta$ and $c_i > c_j$.  

As an example, we see that $a(3) = 16$ thanks to the following colored partitions in question:
$$
(3_3), \ \ \ (3_2), \ \ \ (3_1), \ \ \ (2,1_3), \ \ \ (2,1_2), \ \ \ \ (2,1_1),
$$
$$
(1_3,1_3,1_3), \ \ \ (1_3,1_3,1_2), \ \ \ (1_3,1_3,1_1), 
$$
$$
(1_3,1_2,1_2), \ \ \  (1_3,1_2,1_1), \ \ \  (1_3,1_1,1_1), 
$$
$$
(1_2,1_2,1_2), \ \ \  (1_2,1_2,1_1), \ \ \  (1_2,1_1,1_1), \ \ \  (1_1,1_1,1_1)
$$

For the purposes of this work, we highlight one particular arithmetic property which is satisfied by $a(n)$.  
\begin{theorem}[Corollary 7, \cite{AM}] 
\label{a_72mod7}
For all $n\geq 0$, $a(7n+2) \equiv 0 \pmod{7}$.  
\end{theorem}
We note that Amdeberhan and Merca prove Theorem \ref{a_72mod7} as a corollary of the following theorem which they prove using the Mathematica package RaduRK of Smoot \cite{Smoot}:  
\begin{theorem}
\label{a_7dissection}
We have 
\begin{align*}
        &\sum_{n=0}^\infty a(7n+2)\, q^n \\
        &= 7\left(\frac{1024\,f_2^8\,f_{14}^{18}}{f_1^{20}\,f_7^7}\,q^8 
        +\frac{1344\,f_2^9\,f_{14}^{11}}{f_1^{21}}\,q^6
        -\frac{1024\,f_2^{16}\,f_{14}^{10}}{f_1^{24}\, f_7^3}\,q^5 
        +\frac{72\,f_2^{10}\,f_7^7\,f_{14}^4}{f_1^{22}}\,q^4  \right. \\
         &      \left.   \ \ \ \ \ \ \ 
         -\frac{320\,f_2^{17}\,f_7^4\,f_{14}^3}{f_1^{25}}\,q^3 
        -\frac{40\,f_2^{11}\,f_7^{14}}{f_1^{23}\,f_{14}^3}\,q^2  
        +\frac{56\,f_2^{18}\, f_7^{11}}{f_1^{26}\, f_{14}^4}\,q
        +\frac{f_2^{12}\,f_7^{21}}{f_1^{24}\,f_{14}^{10}} \right).
\end{align*}
\end{theorem}
While this result clearly implies Theorem \ref{a_72mod7}, the technique used to prove Theorem \ref{a_7dissection} is somewhat unsatisfying.  One of the primary goals of this work is to provide a truly elementary proof of Theorem \ref{a_72mod7}.  
(We note in passing that Guadalupe \cite{Gua1} has also provided an elementary proof of Theorem \ref{a_72mod7} using a result from Cooper, Hirschhorn, and Lewis \cite{CHL}.  The proof which we provide below is also elementary, and follows from standard generating function manipulations and theta function results.  Moreover, our proof fits within the context of the other congruence proofs that we will provide in this paper.) 

We now define an infinite family of functions $a_k(n)$ as the number of partitions of $n$ wherein even parts come in only one color, while the odd parts may be ``colored'' with one of $k$ colors for fixed $k\geq 1$.    Clearly, $a_1(n) = p(n)$, the unrestricted integer partition function described above, while $a_3(n) = a(n)$ of Amdeberhan and Merca.  
(We note, in passing, that $a_2(n) = \overline{p}(n)$, the number of overpartitions of weight $n$ \cite{CL, HS05}.)
Using \eqref{genfn2} as a template, we then know that 
\begin{equation}
\label{genfn_general}
\sum_{n=0}^\infty a_k(n)q^n = \frac{f_2^{k-1}}{f_1^k}.
\end{equation}
We can now state our main theorem.
\begin{theorem}
\label{thm:main} 
For all $n\geq 0$, 
\begin{align}
a_{1}(7n+5) & \equiv 0 \pmod{7}, \label{c1} \\
a_{3}(7n+2) & \equiv 0 \pmod{7}, \label{c3} \\
a_{4}(7n+4) & \equiv 0 \pmod{7}, \label{c4} \\
a_{5}(7n+6) & \equiv 0 \pmod{7}, \label{c5} \text{\ \ and}\\
a_{7}(7n+3) & \equiv 0 \pmod{7}. \label{c7} 
\end{align}
\end{theorem}
\noindent 
Note that Theorem \ref{a_72mod7} is \eqref{c3} above.  

As an aside, we briefly highlight the work of Srinivasa Ramanujan on congruence properties satisfied by the partition function $p(n)$ \cite{Ram1919}.  In his groundbreaking work, Ramanujan proved that, for all $n\geq 0$, 
\begin{align}
p(5n+4) &\equiv 0 \pmod{5}, \notag\\
p(7n+5) &\equiv 0 \pmod{7}, \text{\ \ and} \label{RamCongMod7}\\
p(11n+6) &\equiv 0 \pmod{11}.  \notag
\end{align}
(Given the comments above, we see immediately from \eqref{RamCongMod7} that \eqref{c1} holds.)
Readers who are interested in elementary proofs of Ramanujan's congruence results for $p(n)$ mentioned above, as well as generalizations, may wish to consult the work of the first author \cite{Hir}.

In Section \ref{sec:tools}, we collect the tools necessary for proving Theorem \ref{thm:main}.  In Section \ref{sec:proofs}, we then prove Theorem \ref{thm:main}.  All of our proofs are elementary and follow from classic results in $q$--series along with straightforward generating function manipulations.  

\section{Necessary Tools}
\label{sec:tools}
The set of tools that we require to complete each of our proofs is rather small.  First, we require the following well--known classical $q$--series identity.  
\begin{lemma}[Jacobi]
\label{lem:JacobiCube}
We have 
\begin{align*}
f_1^3 
&= \sum_{k=0}^\infty (-1)^k(2k+1)q^{k(k+1)/2} \\
&\equiv \mathcal{J}_0(q^7) + q\mathcal{J}_1(q^7) + q^3\mathcal{J}_3(q^7) \pmod{7}
\end{align*}
where $\mathcal{J}_i$, $i=0,1,3$, are power series with integer coefficients.  
\end{lemma}
\begin{proof}
For an elementary proof of the equality above, see \cite[(1.7.1)]{Hir}.   To obtain the congruence statement above, we follow the argument in \cite[Section 3.4]{Hir}.  Note that 
$$
\frac{k(k+1)}{2} \equiv 0,1,3,6 \pmod{7}
$$ and $k(k+1)/2 \equiv 6 \pmod{7}$  if and only if $k\equiv 3 \pmod{7}$.  In this case, 
the coefficient $2k+1 \equiv 0 \pmod{7}.$  Hence, modulo 7, the only terms that remain must be of the form $q^{7j}, q^{7j+1}$, or $q^{7j+3}$ for some nonnegative integer $j$.  The congruence above follows.  
\end{proof}
\begin{remark}
Note that $\mathcal{J}_0(q^7),  q\mathcal{J}_1(q^7),$ and $q^3\mathcal{J}_3(q^7)$ are the same as $J_0, J_1,$ and $J_3$ in \cite[(3.4.1)]{Hir}.
\end{remark}
We also require three results that appear in \cite[Chapter 10]{Hir}.  
\begin{lemma}
\label{lem:Hir10_identities}
We have 
\begin{align}
\frac{f_1^5}{f_2^2} 
&= \sum_{n=-\infty}^\infty (6n+1)q^{n(3n+1)/2} \notag \\
&\equiv \mathcal{A}_0(q^7) + q\mathcal{A}_1(q^7)+q^5\mathcal{A}_5(q^7) \pmod{7}, \label{Hir10_eqn1}  \\
\frac{f_1^2f_4^2}{f_2} &= \sum_{n=-\infty}^\infty (3n+1)q^{n(3n+2)} \notag \\
&\equiv \mathcal{B}_0(q^7) + q\mathcal{B}_1(q^7)+q^5\mathcal{B}_5(q^7) \pmod{7}, \text{\ \ and} \label{Hir10_eqn2} \\
\frac{f_2^5}{f_1^2} 
&= \sum_{n=-\infty}^\infty (-1)^n(3n+1)q^{n(3n+2)} \notag \\
&\equiv \mathcal{C}_0(q^7) + q\mathcal{C}_1(q^7)+q^5\mathcal{C}_5(q^7) \pmod{7}, \label{Hir10_eqn3}
\end{align}
where each $\mathcal{A}_i$, $\mathcal{B}_i$, and $\mathcal{C}_i$ is a power series with integer coefficients.  
\end{lemma}
\begin{proof}
For proofs of the identities above, see \cite[(10.7.3), (10.7.6), and (10.7.7)]{Hir}, respectively.  

To see the congruences, note first that, for any integer $n$, 
$$n(3n+1)/2 \equiv 0,1,2,5 \pmod{7}.$$  Moreover, $n(3n+1)/2 \equiv 2 \pmod{7}$  if and only if $n\equiv 1 \pmod{7}$.  In this case, 
the coefficient $6n+1 \equiv 0 \pmod{7}.$  Hence, modulo 7, the only terms that remain must be of the form $q^{7j}, q^{7j+1}$, or $q^{7j+5}$ for some nonnegative integer $j$.   This yields the congruence in \eqref{Hir10_eqn1}.  Next, note that, for any integer $n$, 
$$n(3n+2) \equiv 0,1,2,5 \pmod{7}.$$
Moreover, $n(3n+2) \equiv 2 \pmod{7}$  if and only if $n\equiv 2 \pmod{7}$.  In this case, 
the coefficient $3n+1 \equiv 0 \pmod{7}.$  Hence, modulo 7, the only terms that remain must be of the form $q^{7j}, q^{7j+1}$, or $q^{7j+5}$ for some nonnegative integer $j$.   This implies the congruences in both \eqref{Hir10_eqn2} and \eqref{Hir10_eqn3}.  
\end{proof}
To close this section, we note a pivotal congruence result which follows from the Binomial Theorem and congruence properties of certain binomial coefficients.  
\begin{lemma} 
\label{lem:FD} 
For a prime $p$ and positive integers $a$ and $b$, we have 
$$
f_a^{bp} \equiv f_{ap}^b \pmod{p}.  
$$
\end{lemma}

\section{Elementary Proof of Theorem \ref{thm:main} }
\label{sec:proofs}
We can now provide elementary proofs of each congruence in Theorem \ref{thm:main} thanks to the work completed in the previous section.  

\begin{proof}[Proof of Theorem \ref{thm:main}]

\ 

\medskip 

\noindent
Proof of \eqref{c1}.
Note that 
\begin{align*}
& 
\sum_{n=0}^\infty a_1(n)q^n =
\dfrac{1}{f_1} 
=
\dfrac{f_1^6}{f_1^7} 
\equiv
\dfrac{f_1^6}{f_7}  \pmod{7} \text{\ \ from Lemma \ref{lem:FD}} \\
&=
\dfrac{(f_1^3)^2}{f_7} \\
&\equiv 
\dfrac{\left( \mathcal{J}_0(q^7) + q\mathcal{J}_1(q^7) + q^3\mathcal{J}_3(q^7) \right)^2}{f_7}
\pmod{7}
\end{align*}
using Lemma \ref{lem:JacobiCube}.  
Note that it is not possible to obtain a term of the form $q^{7n+5}$ when the last expression above is expanded.  Congruence \eqref{c1} follows.  

As noted earlier, this particular congruence is simply a restatement of Ramanujan's congruence stated in \eqref{RamCongMod7}.  We include the proof above since it fits naturally with the other proofs we provide here.    

\medskip 

\noindent
Proof of \eqref{c3}.
Since 
$$
\sum_{n=0}^\infty a_3(n)q^n
=
\dfrac{f_2^2}{f_1^3}, 
$$
we know that replacing $q$ by $q^2$ immediately yields 
\begin{align*}
&\sum_{n=0}^\infty a_3(n)q^{2n}
=
\dfrac{f_4^2}{f_2^3}
=
\dfrac{f_2^4f_4^2}{f_2^7} 
\\
&\equiv
\dfrac{f_2^4f_4^2}{f_{14}} \pmod{7}  \\ 
&\equiv 
\dfrac{  \left( \mathcal{B}_0(q^7) + q\mathcal{B}_1(q^7)+q^5\mathcal{B}_5(q^7) \right)  \left( \mathcal{C}_0(q^7) + q\mathcal{C}_1(q^7)+q^5\mathcal{C}_5(q^7) \right)     }{f_{14}} \pmod{7}
\end{align*}
using \eqref{Hir10_eqn2} and \eqref{Hir10_eqn3}.  
It is not possible to obtain a term of the form $q^{2(7n+2)}= q^{7(2n)+4}$ in the last expression above, so \eqref{c3} follows.

\medskip 

\noindent
Proof of \eqref{c4}.
We see that 
\begin{align*}
&
\sum_{n=0}^\infty a_4(n)q^n =
\dfrac{f_2^3}{f_1^4} 
=
\dfrac{f_1^3f_2^3}{f_1^7} 
\equiv 
\dfrac{f_1^3f_2^3}{f_7} \pmod{7}   \\ 
&\equiv
\dfrac{ \left(\mathcal{A}_0(q^7) + q\mathcal{A}_1(q^7)+q^5\mathcal{A}_5(q^7)\right) \left(\mathcal{C}_0(q^7) + q\mathcal{C}_1(q^7)+q^5\mathcal{C}_5(q^7)\right)  }{f_7} \pmod{7}
\end{align*}
again using \eqref{Hir10_eqn1} and \eqref{Hir10_eqn3}.   We cannot arrive at a term of the form $q^{7n+4}$ in the last expression above once this expression is expanded, so \eqref{c4} follows.

\medskip 

\noindent
Proof of \eqref{c5}.
In this case, we see that 
\begin{align*}
&
\sum_{n=0}^\infty a_5(n)q^n =
\dfrac{f_2^4}{f_1^5} 
=
\dfrac{f_1^2f_2^4}{f_1^7} 
\equiv
\dfrac{f_1^2f_2^4}{f_7}  \pmod{7}   \\ 
&\equiv 
\dfrac{ \left( \mathcal{A}_0(q^{14}) + q^2\mathcal{A}_1(q^{14})+q^{10}\mathcal{A}_5(q^{14}) \right)  \left( \mathcal{B}_0(q^7) + q\mathcal{B}_1(q^7)+q^5\mathcal{B}_5(q^7) \right) }{f_7} \pmod{7}
\end{align*}
using \eqref{Hir10_eqn1}  (with $q$ replaced by $q^2$) and \eqref{Hir10_eqn2}.  It is not possible to obtain a term of the form $q^{7n+6}$ in the last expression above, so \eqref{c5} follows.  

\medskip 

\noindent
Proof of \eqref{c7}.  
Lastly, 
\begin{align*}
& \sum_{n=0}^\infty a_7(n)q^n 
=
\dfrac{f_2^6}{f_1^7} 
=
\dfrac{(f_2^3)^2}{f_1^7} 
\equiv 
\dfrac{(f_2^3)^2}{f_7}  \pmod{7}   \\ 
&\equiv
\dfrac{\left( \mathcal{J}_0(q^{14}) + q^2\mathcal{J}_1(q^{14}) + q^6\mathcal{J}_3(q^{14}) \right)^2}{f_7} \pmod{7}
\end{align*}
from Lemma \ref{lem:JacobiCube} (with $q$ replaced by $q^2$).   Because we cannot obtain a term of the form $q^{7n+3}$ from the last expression above, we know that the congruence \eqref{c7} follows.  

\end{proof}

\section{Final Remarks}
We close this note with two sets of thoughts.  First, Theorem \ref{thm:main} can be easily extended to an infinite family of results in the following way.  
\begin{corollary}
\label{cor:infinite_families}
For all $j\geq 0$ and all $n\geq 0$, 
\begin{align*}
a_{7j+1}(7n+5) & \equiv 0 \pmod{7},  \\
a_{7j+3}(7n+2) & \equiv 0 \pmod{7}, \\
a_{7j+4}(7n+4) & \equiv 0 \pmod{7},  \\
a_{7j+5}(7n+6) & \equiv 0 \pmod{7}, \text{\ \ and}\\
a_{7j+7}(7n+3) & \equiv 0 \pmod{7}.  
\end{align*}
\end{corollary}
\begin{proof}
Note that for any $j$ and $k$, 
\begin{align*}
\sum_{n=0}^\infty a_{7j+k}(n)q^n
&= 
\frac{f_2^{7j+k-1}}{f_1^{7j+k}} \\
&= 
\frac{f_2^{7j}}{f_1^{7j}}\cdot \frac{f_2^{k-1}}{f_1^{k}} \\
&\equiv 
\frac{f_{14}^{j}}{f_7^{j}}\cdot \frac{f_2^{k-1}}{f_1^{k}} \pmod{7} \\
&= 
\frac{f_{14}^{j}}{f_7^{j}}\sum_{n=0}^\infty a_{k}(n)q^n.
\end{align*}
The result then follows thanks to Theorem \ref{thm:main} and the fact that $\frac{f_{14}^{j}}{f_7^{j}}$ is a function of $q^7$.  
\end{proof}
While Corollary \ref{cor:infinite_families} may not seem impressive given the ease with which it is proved, the fact remains that this corollary provides an infinite collection of congruences modulo 7 for this family of functions.  Moreover, it is not obvious (a priori) from the definition of the partitions in question that such divisibilities modulo 7 will persist as the number of colors on the odd parts increases without bound.  

Second, we note that one might ask what happens when the role of the even parts and the odd parts is switched, that is, if the odd parts are monochromatic and the even parts can appear in $k$ different colors.  In that case, the generating function is given by 
$$
\frac{(q;q^2)_\infty^{k-1}}{(q;q)_\infty^k} = \frac{(q;q)_\infty^{k-1}}{(q;q)_\infty^k(q^2;q^2)_\infty^{k-1}} = \frac{1}{(q;q)_\infty (q^2;q^2)_\infty^{k-1}}.
$$ 
This family of functions has been studied recently; see the work of Amdeberhan, Sellers, and Singh \cite{AmSeSi} as well as subsequent work of Guadalupe \cite{Gua2} and Das, Maity, and Saikia \cite{DMS}.  Indeed, the case $k=2$ of the above family of functions was studied in 2010 by Chan \cite{HCC} who developed a cubic analog of Ramanujan's most beautiful identity.

\end{document}